\documentclass{article}
\usepackage{amsmath, amsthm, amssymb, amsfonts, cite}
\usepackage[center]{titlesec}
\usepackage{lipsum}
\newtheorem{Theorem}
{Theorem}

\newtheorem{Lemma}
{Lemma}
\theoremstyle{remark}

\theoremstyle{Proposition}
\newtheorem*{Proposition}
{Proposition}
\theoremstyle{definition}
\newtheorem{definition}{Definition}[section]

\newcommand{\W}{\mathcal{W}}
\newcommand{\V}{\mathcal{V}}

\newcommand{\N}{\mathbb{N}}

\newcommand{\U}{\mathcal{U}}
\newcommand{\mP}{\mathcal{P}}

\title{Differentiating Infinite Voting Populations using Ultrafilters}
\author{Priyanka Menon\thanks{I am very grateful for the advice and mentorship of Dr. Barry Mazur, Dr. Amartya Sen, Dr. Nate Ackerman, and Dr. Scott Kominers throughout the process of preparing this article.}
 \\Harvard University, Cambridge, MA}
 \date{}

\begin{document}
\maketitle

\mathchardef\mhyphen="2D

\section{Introduction}

Ultrafilters arise frequently in the social choice literature and surrounding fields as collections of the decisive set for a given aggregation procedure. We know from \cite{Kirman} that we can associate each Arrow social welfare function with an ultrafilter. Because of their structure, ultrafilters are useful in analyses of both finite and infinite voting populations.

For an infinite set of voters, \cite{Fishburn} demonstrated that given a society with an infinite population, the non-dictatorship condition of Arrow's general possibility theorem is satisfied, in that there is no single voter whose preferences dictate the outcome of the election. Similar results regarding non-dictatorship in infinite populations have been proven in the fields of judgment aggregation as well \cite{Mongin}. 

Less attention has been paid to the study of the comparison of societies with infinite populations. One of the first efforts to do so was made by \cite{Skala}, utilizing the Rudin-Keisler order over ultrafilters. After this, however, little further work has been undertaken.

In this paper, we focus on the study of the relationship between societies with countably infinite voters. Thus, for a given population, the cardinality of the population of voters in our society is equal to the cardinality of the natural numbers: $|V| = |\N|$.

In this process of comparison, a natural question to ask would be if all sets of decisive voters (ultrafilters) over $\N$ are isomorphic, and thus fundamentally similar. To answer this, we present the following theorem due to \cite{Pospisil}: 
\begin{Theorem}
There are exactly $2^{(2^{\aleph_0})}$ non-principal ultrafilters on $\N$.
\end{Theorem}
\noindent Given this, we see that there are too few permutations for all nonprincipal ultrafilters to be isomorphic, and thus we can begin in earnest the process of comparison.

To this end, in this short paper, we introduce to the social choice literature the notion of the Rudin-Frolik ordering, which allows for the ordering of societies based on the structure of the set of decisive coalitions and is stronger than the Rudin-Keisler ordering used in \cite{Skala}. We then prove a theorem regarding the invisible dictators of \cite{Kirman} using the Rudin-Frolik ordering.

\section{Framework and Definitions}

First, we give the canonical definition of an ultrafilter as it arises in the social choice literature:
\begin{definition}
Take $X$ as a set and $\mP(X)$ as the power set of $X$. We shall call a family of sets, $\mathcal{U} \subseteq \mP(X)$, an \textit{ultrafilter} on $X$ if:
\begin{itemize}
\item $\emptyset \not\in \mathcal{U}$ 
\item $X \in \mathcal{U}$ 
\item If $A \cup B \in \U$, then either $A \in \U$ or $B \in \U$.
\item For $A, B \in \mathcal{U}$, $A \cap B \in \mathcal{U}$
\item If $A \in \mathcal{U}$ and $A \subset B$, then $B \in \mathcal{U}$
\item For any $A \in X$, either $A \in \U$ or $X-A \in \U$.
\end{itemize}
\end{definition}
Closely related to the notion of an ultrafilter is that of a limit along an ultrafilter, which gives a generalized method for taking the limit of sequences in compact spaces:
\begin{definition} 
Take $\U$ as an ultrafilter over $\N$ and take $f: \N \to X$ to be a mapping giving a sequence of elements in a compact space $X$. Then, we can define a limit of $f(n)$ along $\U$ as the unique point, $a$, such that:
$$a = \lim_{n \to \U} f(n) \in X$$
given that if $U$ is an open neighborhood containing $a$, then $f^{-1}(U) \in \U$. Thus, we can think of the sequence as having been indexed by the natural numbers such that the sets of indices $f^{-1}(\U) = \{n \in \N: f(n) \in U\}$ of the elements of $f(n)$ that are mapped into $U$ must be contained in $\U$.
\end{definition}
\noindent We refer readers to \cite{Zirnstein} for a proof of the above mentioned limit's existence and uniqueness. 

Using the generalized limit given to us in the previous definition, we can construct a compactification of the natural numbers. We will use this space, called the Stone-Cech compactification and written as $\beta \N$, as the topological grounding for our investigation of invisible dictatorships.

\begin{definition}
The Stone-Cech compactification of $\N$ is a topological space that has the following properties:
\begin{enumerate}
\item $\N \subset \beta \N$
\item $\beta \N$ is a compact Hausdorff space.
\item $\N$ is dense in $\beta \N$.
\item For any map $f: \N \to Y$, where $Y$ is a compact space, there exists a unique continuous extension of $f$ which maps $\beta \N$ to $Y$, written as $\overline{f}: \beta \N \to Y$.
\end{enumerate}
\end{definition}
We can view $\beta \N$ as the set of all ultrafilters on $\N$, with the principal ultrafilters corresponding to the elements of $\N$. To see the proof of this equivalence, we refer the reader to \cite{Walker}.

\cite{Kirman} situate their investigation and proof of existence of invisible dictatorship in the Stone-Cech compactification of $\N$. Importantly, this means that condition 3 of Definition 2.3 implies that for every situation, $f$, of possible combinations of individual preferences, there exists a continuous extension of $f$ which maps the space of compactified voters to the space of all weak orders, $\mathfrak{P}$, on the set of alternatives, $X$. For this to be the case, however, we must restrict $X$ and $\mathfrak{P}$ to the finite case and then endow $\mathfrak{P}$ with the discrete topology.

Finally, we revise briefly the notation and conditions of Arrow's general possibility theorem, as they appear in \cite{Kirman}:

We write $|X|$ for the number of alternatives, $V$ as the space of voters, $f$ as a function from $V$ to $\mathfrak{P}$ (which again is the set of all weak orders over $X$), $F$ as the set of all $f$, and $\sigma$ (a social welfare function) as a mapping which assigns each $f \in F$ to a weak order $\sigma(f)$ in $\mathfrak{P}$.
\begin{Theorem}
Arrow's general possibility theorem finds that for a finite population, a social welfare function cannot satisfy the following conditions: 
\begin{itemize}
\item $|X| \geq 3$
\item $\sigma$ is a function on $F$ into $\mP$
\item For all $a,b \in X$, $af(V)b \Rightarrow a\sigma(f)b$
\item For all $a,b \in X$ and $f,g \in F$, $f=g$ on $\{a,b\} \Rightarrow \sigma(f) = \sigma(g)$ on $\{a,b\}$
\item There is no $v_0 \in V$ such that , for all $a,b \in X$ and $f \in F$, $af(v_0)b \Rightarrow a\sigma(f)b$
\end{itemize}
\end{Theorem}
As stated earlier, we restrict our attention to the case of countably infinite voters, so $|V|=\N$, and thus (when applicable) write $\N$ to indicate $V$.
\section{Main Results}
\cite{Kirman} defines the concept of invisible dictatorship as a social welfare function, $\sigma$, that violates the following condition:  
\begin{Proposition}
There is no $\overline{\U} \in \beta V$ such that, for all $a, b \in X$ and $f \in F$, $$af(\overline{\U})b \Rightarrow a\sigma(f)b$$
where $\beta V$ is the Stone-Cech compactification of the space of voters, $V$.
\end{Proposition}
In \cite{Kirman}, the invisible dictator is the limit point of the hierarchy that decides the outcome of the social welfare function; thus, the preferences of $\overline{\U}$ are the limit preferences of those in the hierarchy. Formally, we write that $\overline{\U}$ is the limit in $\beta \N$ of an ultrafilter, $\U \in \N$. $\overline{\U}$ is characterized as a point in the compactified space of voters; however, given that we are working in the Stone-Cech compactification of $\N$, we can also characterize $\overline{\U}$ as an ultrafilter over $\N$.

Interpreting $\beta \N$ as the set of all ultrafilters over $\N$, we can partition $\beta \N$ into two sets: dictatorships, corresponding to principal ultrafilters over $\N$, and invisible dictatorships, corresponding to nonprincipal ultrafilters over $\N$.

However, we can differentiate the points of $\beta \N$ further, using the Rudin-Frolik ordering:

\begin{definition}
The \textit{Rudin-Frolik ordering}, denoted by $\sqsubseteq$, is a partial ordering\footnote{We define a partial order on a set as a relation that has the properties of reflexivity, antisymmetry, and transitivity.} on the ultrafilters in $\beta \N$ such that for $\U, \W \in \beta \N$, $\U \sqsubseteq \W$ if and only if there is a one-to-one function, $f: \N \to \beta \N$, such that the set $\{f(n): n \in \N\}$ is discrete\footnote{We define a family of points $S$ as discrete if and only if each point $x \in S$ has a neighborhood $U$ such that the intersection of $S$ and $U$ contains only $\{x\}$.} and $\bar{f}(\U)=\W$.
\end{definition}

Before proving our first theorem, we state a lemma regarding the structure of the Rudin-Frolik ordering:
\begin{definition}
We a define a relation using $\sqsubseteq$ such that, for $\U, \W \in \beta\N$, if $\U \sqsubseteq \W$ and $\W \sqsubseteq \U$, then we write $\U \equiv_{RF} \W$
\end{definition}
\begin{Lemma}
$\equiv_{RF}$ is an equivalence relation on the ultrafilters over $\N$.
\end{Lemma}
\begin{proof}
See \cite{Comfort}.
\end{proof}

First, we demonstrate that the Rudin-Frolik order coincides with the differentiation between points in $\beta \N$ we have already constructed, the differentiation between dictatorial and invisibly dictatorial social welfare functions:

\begin{Theorem}
Those ultrafilters associated with dictatorial social welfare functions are minimal on the Rudin-Frolik ordering
\end{Theorem}
\begin{proof}
As shown in \cite{Kirman}, the ultrafilters that correspond to the decisive coalitions of dictatorial social welfare functions are principal ultrafilters. Thus, we must show that principal ultrafilters are minimal on Rudin-Frolik ordering over $\beta \N$. For proof of this, we refer the readers to \cite{Rudin}.
\end{proof}

The Rudin-Frolik order is useful not only in separating dictatorial aggregation procedures from invisibly dictatorial procedures--it also allows us to differentiate between invisibly dictatorial rules. 
The position of an ultrafilter on the Rudin-Frolik order indicates its ``level of dictatorship," based on its proximity to principal ultrafilters. Intuitively, we suspect that the closer the structure of a (nonprincipal) ultrafilter is to a principal ultrafilter, the more dictatorial it is likely to be. 

We can formalize this intuition with the following discussion of weak P-points, an equivalence class on $\beta \N$ under $\equiv_{RF}$: 

\begin{definition}
Given the topological space $\beta \N$, a point $\V \in \beta \N$ is called a \textit{weak} \textit{P-point} if $\V$ is not contained in the closure of any countable subset of $\beta \N \backslash \{\V\}$.
\end{definition}

Given that we are working in the compact Hausdorff space $\beta \N$, we can translate this definition to imply that if $\V \in \beta \N$ is a weak P-point, then $\V$ is never the limit of a countable (non-trivial) sequence of other ultrafilters in $\beta \N$. 
 
Because we are working in a countably infinite population, the preferences we are concerned with aggregating can only manifest themselves as countable sequences: the invisible dictatorships over $\N$ will only be limit of countably many voters in a hierarchy.

To see formally the structural similarity between principal ultrafilters and weak P-points, consider the process of construction of invisible dictatorships provided in \cite{Kirman}. A given $\sigma$ is first associated with the ultrafilter over $\N$, representing its corresponding set of decisive voters, which is subsequently associated with its unique limit in $\beta \N$. 

\begin{Theorem}
For both principal ultrafilters and weak P-points, if $\overline{\U}$ is a principal ultrafilter or weak P-point, then $\U=\overline{\U}$.
\end{Theorem}
\begin{proof}
Consider first if $\overline{\U}$ is a principal ultrafilter. We then see that $\overline{\U}$ corresponds directly to a decisive set of voters in $\N$, the singleton $\{x\}$ which determines $\overline{\U}$. From \cite{Kirman}, we know $\{x\}$ corresponds directly to a social welfare function, and its corresponding ultrafilter $\U$ over $\N$, and so we have $\overline{\U}=\U$. 

Consider now if $\overline{\U}$ is a weak P-point. We know that $\overline{\U}=\lim_{v \to \U} \U$. Because $\overline{\U}$ is a weak P-point, we know it cannot be the limit of any countable sequence of ultrafilters in $\beta \N$. Because $\U \in \N$, and $\N \subset \beta \N$, then we know that $\overline{\U}$ cannot be the limit of $\U$ in $\beta \N$, unless $\U=\overline{\U}$.
\end{proof}

One could reasonably think that all points contained in $\beta \N \backslash \N$\footnote{ To denote the set of nonprincipal ultrafilters on $\beta\N$, we write $\beta\N \backslash \N$.} are weak P-points; if this is the case, then our revelation concerning ``levels of dictatorship" is not quite as helpful as initially imagined, However, the following theorem demonstrates that there do exist points in $\beta \N$ that are ``less dictatorial" than weak p-points:
\begin{Theorem} \cite{VanMill}
There exists a point $\W \in \beta \N \backslash \N$ such that $\W$ is the limit point of a countable discrete set and the countable sets of which $\W$ is the limit point of comprise a filter. 
\end{Theorem}
\noindent Thus, in the construction of invisible dictatorships, it is not necessarily the case that $\W$ is identified with the ultrafilter over $\N$ that it is the limit of, as is the case with principal ultrafilters and weak P-points.
 
As the above theorem demonstrates, weak P-points are structurally more similar to principal ultrafilters than other ultrafilters in $\beta \N$. Furthermore, on the Rudin-Frolik order, they are the \textit{most} structurally similar nonprincipal ultrafilters to principal ultrafilters in $\beta \N$, thanks to the following theorem:
\begin{Theorem}
The ultrafilters minimal on the Rudin-Frolik ordering over $\beta \N \backslash \N$ are characterized as weak P-points.
\end{Theorem}

\begin{proof}
See \cite{Booth}.
\end{proof}

Finally, the following two theorems allow us to understand the structure of the Rudin-Frolik ordering:
\begin{Theorem}
If $\U \sqsubset \W$, then $\U \neq \W$. 
\end{Theorem}

\begin{Theorem}
\begin{enumerate}
\item $\V \sqsubset \U \sqsubset \W$ implies $\V \sqsubset \W$
\item $\{\U: \U \sqsubset \W\}$ is a linearly ordered set. \footnote{We define a linearly ordered set as a set with a relation on the set such that the relation that has the properties of reflexivity, antisymmetry, transitivity, and comparability.}
\end{enumerate}
\end{Theorem}

\begin{proof}
See \cite{Booth}.
\end{proof}

The linear ordering of the predecessors of an ultrafilter on the Rudin-Frolik order suggests the possibility of constructing a method of quantifying the ``level of dictatorship" of aggregation procedures, based on their proximity on the Rudin-Frolik ordering to dictatorial aggregation procedures. 

Given the one-to-one correspondence between ultrafilters and the set of decisive coalitions of certain aggregation procedures beyond the specific circumstances of social welfare functions, it seems that a potentially fruitful direction of research is to interpret theories concerning ultrafilters in the context of preference aggregation further. Additionally, another avenue for generalization opens: because we present results based on the structural properties of ultrafilters, our results can be applied to domains concerning aggregation outside of social welfare theory in which ultrafilters arise, such as judgment aggregation or fields within computational social choice. Furthermore, though we deal only with situations involving a finite number of alternatives in this paper, given the work of \cite{Grafe}, a suitable generalization to the case of infinite alternatives may be achieved.

\nocite{*}
\bibliography{rf}
\bibliographystyle{plain}
\end{document}